\documentclass[11pt]{article}
\usepackage{latexsym,amssymb,amsmath,amsthm,enumerate,geometry,float,cite}
\newtheorem{theorem}{Theorem}
\newtheorem{lemma}[theorem]{Lemma}

\newtheorem{claim}{Claim}
\usepackage{lineno}
\usepackage{setspace}
\usepackage{showlabels}

\begin{document}
\onehalfspace

\title{Bounding the Mostar index}
\author{
\v{S}tefko Miklavi\v{c}$^1$\and 
Johannes Pardey$^2$\and 
Dieter Rautenbach$^2$\and 
Florian Werner$^2$}
\date{}

\maketitle
\vspace{-10mm}
\begin{center}
{\small 
$^1$ University of Primorska, Institute Andrej Maru\v{s}i\v{c}, Koper, Slovenia\\
\texttt{stefko.miklavic@upr.si}\\[3mm]
$^2$ Institute of Optimization and Operations Research, Ulm University, Ulm, Germany\\
\texttt{$\{$johannes.pardey,dieter.rautenbach,florian.werner$\}$@uni-ulm.de}
}
\end{center}

\begin{abstract}
Do\v{s}li\'{c} et al.~defined the Mostar index of a graph $G$ 
as $Mo(G)=\sum\limits_{uv\in E(G)}|n_G(u,v)-n_G(v,u)|$,
where, for an edge $uv$ of $G$,
the term $n_G(u,v)$ denotes the number of vertices of $G$
that have a smaller distance in $G$ to $u$ than to $v$.
They conjectured that $Mo(G)\leq 0.\overline{148}n^3$
for every graph $G$ of order $n$.
As a natural upper bound on the Mostar index,
Geneson and Tsai implicitly consider the parameter 
$Mo^\star(G)=\sum\limits_{uv\in E(G)}\big(n-\min\{ d_G(u),d_G(v)\}\big)$.
For a graph $G$ of order $n$,
they show that $Mo^\star(G)\leq \frac{5}{24}(1+o(1))n^3$.

We improve this bound to $Mo^\star(G)\leq \left(\frac{2}{\sqrt{3}}-1\right)n^3$,
which is best possible up to terms of lower order.
Furthermore, 
we show that 
$Mo^\star(G)\leq 
\left(2\left(\frac{\Delta}{n}\right)^2+\left(\frac{\Delta}{n}\right)-2\left(\frac{\Delta}{n}\right)\sqrt{\left(\frac{\Delta}{n}\right)^2+\left(\frac{\Delta}{n}\right)}\right)n^3$
provided that $G$ has maximum degree $\Delta$.\\[3mm]
\end{abstract}

\section{Introduction}

Do\v{s}li\'{c} et al.~\cite{domasktizu} defined the {\it Mostar index} $Mo(G)$ 
of a (finite and simple) graph $G$ as
$$Mo(G)=\sum\limits_{uv\in E(G)}|n_G(u,v)-n_G(v,u)|,$$
where, for an edge $uv$ of $G$,
the term $n_G(u,v)$ denotes the number of vertices of $G$
that have a smaller distance in $G$ to $u$ than to $v$.
Since its introduction in 2018 
the Mostar index has already incited a lot of research,
mostly concerning sparse graphs and trees  {\cite{AXK, DL2, DL3, HZ1, HZ2, Tepeh}, chemical graphs \cite{CLXZ, DL1, GA, GR, XZTHD}, and hypercube-related graphs \cite{Mollard, OSS}, see also the recent survey \cite{aldo}.

Do\v{s}li\'{c} et al.~\cite{domasktizu} conjectured that $S_{n/3,2n/3}$ 
has maximum Mostar index 
among all graphs of order $n$
(cf. \cite[Conjecture 20]{domasktizu}),
where $n$ is a multiple of $3$, and 
$S_{k,n-k}$ denotes the split graph that arises from the disjoint
union of a clique $C$ of order $k$ and an independent set $I$ of order $n-k$
by adding all possible edges between $C$ and $I$.
Note that 
$$Mo(S_{n/3,2n/3})=\frac{4}{27}(1-o(1))n^3=0.\overline{148}(1-o(1))n^3.$$
As observed in \cite{domasktizu}
the Mostar index of a graph $G$ of order $n$ is less than $\frac{n^3}{2}$;
each of its less than $\frac{n^2}{2}$ edges
contributes less than $n$ to $Mo(G)$.
Geneson and Tsai \cite{gets} improved this trivial upper bound 
to $\frac{5}{24}(1+o(1))n^3\approx 0.2083(1+o(1))n^3$.
They actually show this upper bound for the parameter
$$Mo^\star(G)=\sum\limits_{uv\in E(G)}\big(n-\min\{ d_G(u),d_G(v)\}\big),$$
where $d_G(u)$ denotes the degree of a vertex $u$ in $G$.
The parameter $Mo^\star(G)$ is a natural upper bound on $Mo(G)$:
If $uv$ is an edge of $G$ with $n_G(u,v)\geq n_G(v,u)$, 
then 
$n_G(v,u)\geq |\{ v\}|=1$ 
and
$n_G(u,v)\leq |V(G)\setminus (N_G[v]\setminus \{ u\})|=n-d_G(v)$,
where $N_G[v]$ denotes the closed neighborhood of $v$ in $G$.
We obtain
$$
|n_G(u,v)-n_G(v,u)|
=n_G(u,v)-n_G(v,u)
\leq n-1-d_G(v)
<n-\min\big\{ d_G(u),d_G(v)\big\},$$
and, hence, 
$$\mbox{$Mo(G)\leq Mo^\star(G)$ for every graph $G$.}$$

As our first main result, we prove the following.

\begin{theorem}\label{theorem3}
If $G$ is a graph of order $n$, then 
$Mo^\star(G)\leq \left(\frac{2}{\sqrt{3}}-1\right)n^3\leq 0.1548n^3.$
\end{theorem}
Theorem \ref{theorem3} is best possible up to terms of lower order,
which means that we cannot prove Conjecture 20 from \cite{domasktizu}
by considering only $Mo^\star(G)$:
Starting with a complete bipartite graph whose smaller partite set contains about a 
$\gamma=\frac{\sqrt{3}-1}{2}\approx 0.366$ fraction of all vertices,
and 
recursively inserting in that smaller partite set 
a further complete bipartite graph whose smaller partite set contains about a $\gamma$ fraction of its vertices
yields a recursive construction of a graph $G$ of order $n$ with 
$Mo^\star(G)=\left(\frac{2}{\sqrt{3}}-1\right)(1-o(1))n^3$.
Note that the complement of the constructed graph 
is the disjoint union of cliques of approximate orders 
$\gamma^i(1-\gamma)n$ for $i=0,1,2,3,\ldots$.
Inspecting the proof of Theorem \ref{theorem3} reveals that the above recursive
construction of the (approximately) extremal graphs is quite natural for $Mo^\star$,
which is a rather unusual and mathematically pleasing feature of this new parameter.

Our second main result relies on a linear programming approach 
that we introduced in \cite{miparawe},
where we determined essentially best possible upper bounds 
on the Mostar index of bipartite graphs and split graphs.

\begin{theorem}\label{theorem1}
If $G$ is a graph of order $n$ and maximum degree $\Delta$, then 
$$Mo^\star(G)\leq 
\left(2\left(\frac{\Delta}{n}\right)^2+\left(\frac{\Delta}{n}\right)-2\left(\frac{\Delta}{n}\right)\sqrt{\left(\frac{\Delta}{n}\right)^2+\left(\frac{\Delta}{n}\right)}\right)n^3.$$
\end{theorem}
The bound in Theorem \ref{theorem1} is increasing in $\frac{\Delta}{n}$
and improves Theorem \ref{theorem3} only for $\Delta$ up to about $0.725n$.
Before we proceed to Section \ref{section2}, 
where we prove our two theorems,
we discuss further results and possible approaches.

It is easy to see that 
$Mo(G)=irr(G)$ for graphs $G$ of order $n$ and diameter at most $2$,
where 
$$irr(G)=\sum\limits_{uv\in E(G)}|d_G(u)-d_G(v)|$$
is the {\it irregularity} introduced by Albertson \cite{al}.
As he showed that $irr(G)\leq \frac{4}{27}n^3$, 
Conjecture 20 from \cite{domasktizu} 
follows immediately for graphs of diameter at most $2$.
This suggests the approach to show that, 
for every order $n$,
some graph $G$ maximizing the Mostar index 
among all graphs of order $n$
has a universal vertex.
This, in turn, suggests considering the effect on the Mostar index 
of adding missing edges to some given graph.
Unfortunately, adding edges can have considerable non-local effects 
on the contribution of individual edges to the Mostar index.
Nevertheless, examples suggest that suitable missing edges 
whose addition to a given graph $G$ 
might have a controllable effect on the Mostar index of $G$
can be identified using the following partial orientation of $G$:
$$\mbox{\it For every edge $uv$ of $G$ with $n_G(u,v)>n_G(v,u)$, orient $uv$ from $v$ to $u$.}$$
This orientation is {\it acyclic} in the following sense:
As observed in \cite{aoha,dogu}, 
we have $n_G(u,v)-n_G(v,u)=\sigma_G(v)-\sigma_G(u)$
for every edge $uv$ of $G$, where 
$\sigma_G(x)=\sum\limits_{y\in V(G)}{\rm dist}_G(x,y)$,
and ${\rm dist}_G(x,y)$ denotes the distance in $G$ between the vertices $x$ and $y$.
Now, if $C:u_1u_2\ldots u_\ell u_1$ is a cycle in $G$
such that
the edge $u_\ell u_1$ is oriented from $u_\ell$ to $u_1$ and,
for every $i\in [\ell-1]$,  
the edge $u_iu_{i+1}$ is either oriented from $u_i$ to $u_{i+1}$ or is not oriented at all, then 
$\sum\limits_{uv\in E(C)}|n_G(u,v)-n_G(v,u)|$ should be strictly positive, yet
\begin{eqnarray*}
\sum\limits_{uv\in E(C)}|n_G(u,v)-n_G(v,u)| & = &  
\sum\limits_{i=1}^{\ell-1}\Big(n_G(u_{i+1},u_i)-n_G(u_i,u_{i+1})\Big)+\Big(n_G(u_1,u_\ell)-n_G(u_\ell,u_1)\Big)\\
&=&\sum\limits_{i=1}^{\ell-1}\Big(\sigma_G(u_i)-\sigma_G(u_{i+1})\Big)+\Big(\sigma_G(u_{\ell})-\sigma_G(u_1)\Big)=0,
\end{eqnarray*}
that is, no such cycle exists in $G$.
We believe that adding missing edges 
between vertices of zero outdegree and vertices of zero indegree 
in this partial orientation might have a controllable effect on the Mostar index of $G$.
Unfortunately, we have not been able to quantify this intuition sufficiently well.

The approach of Geneson and Tsai \cite{gets} 
actually allows to obtain upper bounds on $Mo^\star(G)$
depending on the degree sequence of $G$:
Let the graph $G$ have $n$ vertices, $m$ edges, 
and vertex degrees $d_1\leq d_2\leq \ldots \leq d_n$.
Let $V(G)=\{ u_1,\ldots,u_n\}$ be such that $d_G(u_i)=d_i$ 
for every $i\in \{ 1,2,\ldots,n\}$.
Furthermore, for every such $i$, 
let $e_i$ be the number of neighbors of $u_i$ in $\{ u_{i+1},\ldots,u_n\}$.
Now, $m=\sum\limits_{i=1}^ne_i$, and 
$Mo^\star(G)=\sum\limits_{i=1}^ne_i(n-d_i)
=nm-\sum\limits_{i=1}^ne_id_i.$
Clearly, for every $i$, we have 
$$\max\{ 0,d_i-i+1\}=:e_i^-\leq e_i\leq e_i^+:=\min\{ d_i,n-i\},$$
which easily implies
$\sum\limits_{i=1}^ne_id_i\geq 
s:=\sum\limits_{i=1}^{k-1}e^+_id_i+\sum\limits_{i=k}^ne^-_id_i$,
where $k$ is the smallest integer with 
$m\leq \sum\limits_{i=1}^ke^+_i+\sum\limits_{i=k+1}^ne^-_i$.
Altogether, we obtain 
$Mo^\star(G)\leq nm-s$,
where $n$, $m$, and $s$ only depend on the degree sequence of $G$.
The Mostar index of trees of given degree sequences has been studied in \cite{DL2}.

\section{Proofs of Theorems \ref{theorem3} and \ref{theorem1}}\label{section2}

In the present section we prove our two main results.

\begin{proof}[Proof of Theorem \ref{theorem3}]
The proof is by induction in $n$.
For $n=1$, the graph $G$ has no edge, $Mo^\star(G)=0$, and the statement is trivial.
Now, let $n>1$, and let the graph $G$ of order $n$ be chosen in such a way that
\begin{enumerate}[(i)]
\item $Mo^\star(G)$ is as large as possible,
\item subject to (i), the graph $G$ has as many edges as possible, and
\item subject to (i) and (ii), 
the term $\sum\limits_{u\in V(G)}d^2_G(u)$ is as large as possible.
\end{enumerate}
For a linear ordering $\pi:u_1,u_2,\ldots,u_n$ of the vertices of $G$, 
an edge $u_iu_j$ with $i<j$ is called a {\it forward edge at $u_i$}.
For $i\in [n]$, let $d^+_i$ be the number of forward edges at $u_i$.
Note that $d_i^+$ depends on the specific choice of $\pi$.

Now, choose $\pi:u_1,u_2,\ldots,u_n$ such that 
\begin{enumerate}
\item[(iv)] $d_G(u_1)\leq d_G(u_2)\leq \ldots \leq d_G(u_n)$, and 
\item[(v)] subject to (iv), 
the term $w(\pi)=\sum\limits_{i=1}^n (n-i)d^+_i$ is as large as possible.
\end{enumerate}

\begin{claim}\label{claim1}
If $d_G(u_i)=d_G(u_{i+1})$ for some $i\in [n-1]$, 
then $d^+_i\geq d^+_{i+1}$.
\end{claim}
\begin{proof}[Proof of Claim \ref{claim1}]
Suppose, for a contradiction, that $d^+_i<d^+_{i+1}$.
Let the linear ordering $\pi'$ arise from $\pi$ by exchanging $u_i$ and $u_{i+1}$.
If $u_i$ and $u_{i+1}$ are not adjacent, then 
\begin{eqnarray*}
w(\pi')-w(\pi)&=& 
(n-i)d^+_{i+1}+(n-i-1)d^+_i
-(n-i)d^+_{i}-(n-i-1)d^+_{i+1}\\
&=&d^+_{i+1}-d^+_i>0,
\end{eqnarray*}
and, if $u_i$ and $u_{i+1}$ are adjacent, then 
\begin{eqnarray*}
w(\pi')-w(\pi)&=& 
(n-i)(d^+_{i+1}+1)+(n-i-1)(d^+_i-1)
-(n-i)d^+_{i}-(n-i-1)d^+_{i+1}\\
&=&d^+_{i+1}-d^+_i+1>0.
\end{eqnarray*}
Since $\pi'$ satisfies (iv), 
we obtain a contradiction to condition (v) in the choice of $\pi$.
\end{proof}

\begin{claim}\label{claim2}
If $d_G(u_i)<d_G(u_{i+1})$ for some $i\in [n-1]$, 
then $d^+_i\geq d^+_{i+1}$.
\end{claim}
\begin{proof}[Proof of Claim \ref{claim2}]
Suppose, for a contradiction, that $d^+_i<d^+_{i+1}$.
This implies the existence of a forward edge $u_{i+1}u_j$ at $u_{i+1}$
for which $u_i$ is not adjacent to $u_j$.

Let $G'=G-u_{i+1}u_j+u_iu_j$.

In order to lower bound $Mo^\star(G')-Mo^\star(G)$, 
we consider the contributions of the different edges.
\begin{itemize}
\item The edge $u_{i+1}u_j$ of $G$ contributes $n-d_G(u_{i+1})$ to $Mo^\star(G)$.
\item The edge $u_iu_j$ of $G'$ contributes $n-(d_G(u_i)+1)$ to $Mo^\star(G')$.
\item Each of the $d_{i+1}^+-1$ forward edges of $G$ at $u_{i+1}$ 
that are distinct from $u_{i+1}u_j$
contributes one more to $Mo^\star(G')$ than to $Mo^\star(G)$.
\item Each of the $d_i^+$ forward edges of $G$ at $u_i$ 
contributes at most one less to $Mo^\star(G')$ than to $Mo^\star(G)$.
Note that, if $u_i$ and $u_{i+1}$ are adjacent, 
then the edge between them is one of these forward edges.
\item All remaining edges contribute at least as much to $Mo^\star(G')$ as to $Mo^\star(G)$.
\end{itemize}
Since $Mo^\star(G')\leq Mo^\star(G)$ by the choice of $G$, these observations imply
\begin{eqnarray*}
0 &\geq & Mo^\star(G')-Mo^\star(G) \\
& \geq &  
-(n-d_G(u_{i+1}))
+(n-(d_G(u_i)+1))
+(d_{i+1}^+-1)
-d_i^+\\
& = & (\underbrace{d_G(u_{i+1})-d_G(u_i)-1}_{\geq 0})+(\underbrace{d_{i+1}^+-d_i^+-1}_{\geq 0})\geq 0,
\end{eqnarray*}
and, hence,
\begin{eqnarray}
d_G(u_{i+1})&=&d_G(u_i)+1\label{ef0a}\mbox{ and }\\
d_{i+1}^+&=&d_i^++1.\label{ef0b}
\end{eqnarray}
If $u_i$ and $u_{i+1}$ are adjacent, then, by (\ref{ef0a}), 
the forward edge $u_iu_{i+1}$ at $u_i$
contributes the same to $Mo^\star(G')$ as to $Mo^\star(G)$,
which implies the contradiction $Mo^\star(G')-Mo^\star(G)>0$.
Hence, the vertices $u_i$ and $u_{i+1}$ are not adjacent.

Let $G^+=G+u_iu_j$ and $G^-=G-u_{i+1}u_j$.

In order to lower bound $Mo^\star(G^+)-Mo^\star(G)$, 
we consider the contributions of the different edges.
\begin{itemize}
\item The edge $u_iu_j$ of $G^+$ contributes $n-(d_G(u_i)+1)$ to $Mo^\star(G^+)$.
\item Each of the $d_i^+$ forward edges of $G$ at $u_i$ 
contributes one less to $Mo^\star(G^+)$ than to $Mo^\star(G)$.
\item Each of the $d_j^+$ forward edges of $G$ at $u_j$ 
contributes at most one less to $Mo^\star(G^+)$ than to $Mo^\star(G)$.
\item All remaining edges contribute at least as much to $Mo^\star(G^+)$ as to $Mo^\star(G)$.
\end{itemize}
Together, these observations imply
\begin{eqnarray}\label{ef1}
Mo^\star(G^+)-Mo^\star(G) & \geq & n-d_G(u_i)-1-d_i^+-d_j^+.
\end{eqnarray}
In order to upper bound  $Mo^\star(G)-Mo^\star(G^-)$, 
we consider the contributions of the different edges.
\begin{itemize}
\item The edge $u_{i+1}u_j$ of $G$ contributes $n-d_G(u_{i+1})\stackrel{(\ref{ef0a})}{=}n-(d_G(u_i)+1)$ to $Mo^\star(G)$.
\item Each of the $d_{i+1}^+-1\stackrel{(\ref{ef0b})}{=}d_i^+$ forward edges of $G$ at $u_{i+1}$
that are distinct from $u_{i+1}u_j$
contributes one less to $Mo^\star(G)$ than to $Mo^\star(G^-)$.
\item Each of the $d_j^+$ forward edges of $G$ at $u_j$ 
contributes one less to $Mo^\star(G)$ than to $Mo^\star(G^-)$.
\item All remaining edges contribute at most as much to $Mo^\star(G)$ as to $Mo^\star(G^-)$.
\end{itemize}
Together, these observations imply
\begin{eqnarray}\label{ef2}
Mo^\star(G)-Mo^\star(G^-) & \leq & n-d_G(u_i)-1-d_i^+-d_j^+.
\end{eqnarray}
Combining (\ref{ef1}) and (\ref{ef2}), and using the specific choice of $G$, 
we obtain the contradiction
$$0>Mo^\star(G^+)-Mo^\star(G)\geq Mo^\star(G)-Mo^\star(G^-)\geq 0.$$
\end{proof}
By Claims \ref{claim1} and \ref{claim2}, we have
$d_1^+\geq d_2^+\geq \ldots \geq d_n^+$.

\begin{claim}\label{claim3}
If $u_i$ and $u_j$ are adjacent for some $1\leq i<j\leq n-1$, 
then $u_i$ is adjacent to $u_j,u_{j+1},\ldots,u_n$.
\end{claim}
\begin{proof}[Proof of Claim \ref{claim3}]
Suppose, for a contradiction, that $u_i$ is adjacent to $u_j$ but not to $u_{j+1}$ 
for some $1\leq i<j\leq n-1$.

Let $G'=G-u_iu_j+u_iu_{j+1}$.

In order to lower bound $Mo^\star(G')-Mo^\star(G)$, 
we consider the contributions of the different edges.
\begin{itemize}
\item Each of the $d_j^+$ forward edges of $G$ at $u_j$ 
contributes one more to $Mo^\star(G')$ than to $Mo^\star(G)$.
\item Each of the $d_{j+1}^+$ forward edges of $G$ at $u_{j+1}$ 
contributes at most one less to $Mo^\star(G')$ than to $Mo^\star(G)$.
\item All remaining edges contribute at least as much to $Mo^\star(G')$ as to $Mo^\star(G)$.
\end{itemize}
Since $Mo^\star(G')\leq Mo^\star(G)$ by the choice of $G$, these observations imply
\begin{eqnarray*}
0 &\geq & Mo^\star(G')-Mo^\star(G)\geq d_j^+-d_{j+1}^+\geq 0,
\end{eqnarray*}
that is, we have $Mo^\star(G')=Mo^\star(G)$.
Note that $G'$ has the same number of edges as $G$ 
but that 
\begin{eqnarray*}
\sum\limits_{u\in V(G)}d^2_{G'}(u)-\sum\limits_{u\in V(G)}d^2_G(u)
&=&(d_G(u_j)-1)^2+(d_G(u_{j+1})+1)^2-d_G(u_j)^2-d_G(u_{j+1})^2\\
&=& 2(d_G(u_{j+1})-d_G(u_j))+2>0,
\end{eqnarray*}
which implies a contradiction to condition (iii) in the choice of $G$.
\end{proof}
Let $\delta=d_G(u_1)$.
By Claim \ref{claim3},
the neighborhood of $u_1$ in $G$ is $V(G)\setminus I$, where $I=\{ u_1,\ldots,u_{n-\delta}\}$.
If $u_iu_j$ is an edge of $G$ for $1\leq i<j\leq n-\delta$,
then Claim \ref{claim3} implies
$d_i^+\geq n-j+1\geq \delta+1$, which implies the contradiction
$d_G(u_1)=d_1^+\geq d_i^+\geq \delta+1$.
Hence, the set $I$ is independent.
Since each vertex in $I$ has degree at least $\delta$,
and $V\setminus I$ contains exactly $\delta$ vertices,
it follows that each vertex in $I$ has degree exactly $\delta$,
and that there are all possible $\delta(n-\delta)$ edges in $G$ between $I$ and $V\setminus I$.

Let $H=G-I$ and $x=\frac{\delta}{n}$.

Using 
$$\delta-\min\{ d_H(u),d_H(v)\}=n-\min\{ d_G(u),d_G(v)\}\mbox{ for every edge $uv$ of $H$,}$$  
induction, and
$$\max\left\{x(1-x)^2+\left(\frac{2}{\sqrt{3}}-1\right)x^3:x\in [0,1]\right\}=\frac{2}{\sqrt{3}}-1,$$
we obtain 
\begin{eqnarray*}
Mo^\star(G) &=& \delta(n-\delta)^2+Mo^\star(H)\\
&\leq &\delta(n-\delta)^2+\left(\frac{2}{\sqrt{3}}-1\right)\delta^3\\
&= &\left(x(1-x)^2+\left(\frac{2}{\sqrt{3}}-1\right)x^3\right)n^3\\
& \leq & \left(\frac{2}{\sqrt{3}}-1\right)n^3,
\end{eqnarray*}
which completes the proof.
\end{proof}
We proceed to the proof of Theorem \ref{theorem1}.
Let $G$ be a graph of order $n$ and maximum degree at most $\Delta$
for positive integers $n$ and $\Delta$ with $\Delta\leq n-1$.

Let $I=\{ 0,1,\ldots,\Delta\}$, and let $G$ have 
\begin{itemize} 
\item $x_i n$ vertices of degree $i$ for every $i\in I$, and
\item $y_{i,j} n^2$ edges between vertices of degree $i$ and vertices of degree $j$ for every $i,j\in I$ with $i\leq j$.
\end{itemize} 
Double-counting the edges incident with vertices of degree $i$ in $G$ implies
$$i x_i n=2y_{i,i}n^2+\sum\limits_{j\in I:j<i}y_{j,i}n^2+\sum\limits_{j\in I:i<j}y_{i,j}n^2.$$
We obtain
\begin{eqnarray}\label{e2}
Mo^\star(G)&\leq & \sum\limits_{i,j\in I:i\leq j}\left(n-i\right)y_{i,j}n^2 \leq {\rm OPT}(P)n^3,
\end{eqnarray}
where ${\rm OPT}(P)$ denotes the optimal value 
of the following linear programm $(P)$:
$$\begin{array}{rrrcll}
& \max  & \sum\limits_{i,j\in I:i\leq j}\left(1-\frac{i}{n}\right)y_{i,j}&&&\\
(P) \,\,\,\,\,\,\,\,\, & s.th.  & \sum\limits_{i\in I}x_i&=&1,&\\
&   & 2y_{i,i}+\sum\limits_{j\in I:j<i}y_{j,i}+\sum\limits_{j\in I:i<j}y_{i,j}-\frac{i}{n}x_i &=&0&
\mbox{ for every $i\in I$, and}\\
&   & x_i,y_{i,j} &\geq &0& \mbox{ for every $i,j\in I$ with $i\leq j$}.
\end{array}$$
The dual of $(P)$ is the following linear programm $(D)$:
$$
\begin{array}{rrrcll}
& \min & p &  &  &\\
(D) \,\,\,\,\,\,\,\,\, & s.th. & q_i+q_j & \geq & 1-\frac{i}{n} & \mbox{ for every $i,j\in I$ with $i\leq j$},\\
& & p & \geq & \frac{i}{n}q_i & \mbox{ for every $i\in I$, and}\\
& & p,q_i & \in & \mathbb{R} & \mbox{ for every $i\in I$}.
\end{array}
$$
For our argument, we actually only need the 
{\it weak duality inequality chain} for $(P)$ and $(D)$,
which holds for all pairs of feasible solutions of $(P)$ and $(D)$:
\begin{eqnarray*}
\sum\limits_{i,j\in I:i\leq j}\left(1-\frac{i}{n}\right)y_{i,j}
& \leq &
\sum\limits_{i,j\in I:i\leq j}(q_i+q_j)y_{i,j}
+\sum\limits_{i\in I}\left(p-\frac{i}{n}q_i\right)x_i\\
& = &
\sum\limits_{i\in I}x_ip
+\sum\limits_{i\in I}q_i\left(2y_{i,i}+\sum\limits_{j\in I:j<i}y_{j,i}+\sum\limits_{j\in I:i<j}y_{i,j}-\frac{i}{n}x_i\right)\\
& = & p.
\end{eqnarray*}
Theorem \ref{theorem1} follows by combining (\ref{e2}), weak duality ${\rm OPT}(P)\leq {\rm OPT}(D)$, and the following lemma.

\begin{lemma}\label{lemma1}
${\rm OPT}(D)\leq p_{\Delta}$ for $p_{\Delta}=2\left(\frac{\Delta}{n}\right)^2+\left(\frac{\Delta}{n}\right)-2\left(\frac{\Delta}{n}\right)\sqrt{\left(\frac{\Delta}{n}\right)^2+\left(\frac{\Delta}{n}\right)}$.
\end{lemma}
\begin{proof}
Since $(D)$ is a minimization problem, it suffices to provide a feasible solution with objective function value $p_{\Delta}$.
Therefore, let $p=p_{\Delta}$, $q_0=1$, and $q_i=\frac{n}{i}p$ for every $i\in I\setminus \{ 0\}$.
Note that this ensures the dual constraint {\it $p\geq\frac{i}{n}q_i$ for every $i\in I$}.
Since $q_i$ is decreasing for $i\geq 1$, and $\frac{i}{n}\in \left[0,\frac{\Delta}{n}\right]$,
we have $q_j\geq \frac{n}{\Delta}p$ for $j\in I$, and the dual constraint
\begin{eqnarray}\label{e3}
\mbox{\it $q_i+q_j \geq 1-\frac{i}{n}$ for every $i,j\in I$ with $i\leq j$}
\end{eqnarray}
holds provided that 
\begin{eqnarray}\label{e4}
\mbox{$\frac{p}{x}+x+\frac{n}{\Delta}p\geq 1$ for $x\in \left(0,\frac{\Delta}{n}\right]$.}
\end{eqnarray}
The function $x\mapsto \frac{p}{x}+x$ for $x\in (0,\infty)$ assumes its minimum of $2\sqrt{p}$ at $\sqrt{p}$.
Hence, (\ref{e4}) holds provided that 
$2\sqrt{p}+\frac{n}{\Delta}p=1$.
It is easy to verify that this holds indeed for $p=p_{\Delta}$, which completes the proof.
\end{proof}


\begin{thebibliography}{}
\bibitem{al} M. Albertson, The irregularity of a graph, {\it Ars Combin.} {\bf 46} (1997) 219--225.

\bibitem{aldo} 
A. Ali and T. Do\v{s}li\'{c}, Mostar index: results and perspectives, {\em Appl. Math. Comput.} {\bf 404} (2021) Paper No. 126245.

\bibitem{AXK}
Y. Alizadeh, K. Xu, and S. Klav\v{z}ar,  On the Mostar index of trees and product graphs, {\em Filomat} {\bf 35} (2021) 4637--4643.

\bibitem{aoha}
M. Aouchiche and P. Hansen, On a conjecture about Szeged index, {\it European J. Combin.} {\bf 31} (2010) 1662--1666.

\bibitem{CLXZ}
H. Chen, H. Liu, Q. Xiao, and J. Zhang, Extremal phenylene chains with respect to the Mostar index, {\em Discrete Math. Algorithms Appl.} {\bf 13} (2021) Paper No. 2150075.

\bibitem{DL1}
K. Deng and S. Li, Extremal catacondensed benzenoids with respect to the Mostar index, {\em J. Math. Chem.} {\bf 58} (2020) 1437--1465.

\bibitem{DL2} 
K. Deng and S. Li, On the extremal values for the Mostar index of trees with given degree sequence, {\em Appl. Math. Comput.} {\bf 390} (2021) Paper No. 125598.

\bibitem{DL3} 
K. Deng and S. Li, On the extremal Mostar indices of trees with a given segment sequence, {\em Bull. Malays. Math. Sci. Soc.} {\bf 45} (2022) 593--612.

\bibitem{domasktizu}
T. Do\v{s}li\'{c}, I. Martinjak, R. \v{S}krekovski, S. Tipuri\'{c} Spu\v{z}evi\'{c}, and I. Zubac, Mostar index, {\em  J. Math. Chem.} {\bf 56} (2018) 2995--3013.

\bibitem{dogu}
A. Dobrynin and I. Gutman, On a graph invariant related to the sum of all distances in a graph, {\em Publ. Inst. Math. (Beograd) (N.S.)} {\bf 56} (1994) 18--22.

\bibitem{erga} P. Erd\H{o}s and T. Gallai, Graphs with prescribed degrees of vertices (in Hungarian), {\it Mat. Lapok} {\bf 11} (1960) 264--274.

\bibitem{OSS}
\"{O}. E\u{g}ecio\u{g}lu, E. Sayg\i, and Z. Sayg\i, The Mostar index of Fibonacci and Lucas cubes, {\em Bull. Malays. Math. Sci. Soc.} {\bf 44} (2021) 3677--3687.

\bibitem{gets} 
J. Geneson and S.-F. Tsai, Peripherality in networks: theory and applications,
{\em  J. Math. Chem.} {\bf 60} (2022) 1021--1079.

\bibitem{GA}
N. Ghanbari and S. Alikhani, Mostar index and edge Mostar index of polymers, {\em Comput. Appl. Math.} {\bf 40} (2021) Paper No. 260.

\bibitem{GR}
M. Ghorbani and S. Rahmani, The Mostar index of fullerenes in terms of automorphism group, {\em Facta Univ. Ser. Math. Inform.} {\bf 35} (2020) 151--165.

\bibitem{HZ1}
F. Hayat and B. Zhou, On cacti with large Mostar index, {\em Filomat} {\bf 33} (2019) 4865--4873.

\bibitem{HZ2}
F. Hayat and B. Zhou, On Mostar index of trees with parameters, {\em Filomat} {\bf 33} (2019) 6453--6458.

\bibitem{miparawe} 
\v{S}. Miklavi\v{c}, J. Pardey, D. Rautenbach, and F. Werner, Maximizing the Mostar index for bipartite graphs and split graphs, arXiv 2210.03399.

\bibitem{Mollard}
M. Mollard, A relation between Wiener index and Mostar index for daisy cubes, {\em Discrete Math. Lett.} {\bf 10} (2022) 81--84.

\bibitem{Tepeh}
A. Tepeh, Extremal bicyclic graphs with respect to Mostar index, {\em Appl. Math. Comput.} {\bf 355} (2019) 319--324.

\bibitem{XZTHD}
Q. Xiao, M. Zeng, Y. Tang, H Hua, and H. Deng, The hexagonal chains with the first three maximal Mostar indices, {\em Discrete Appl. Math.} {\bf 288} (2021) 180--191.
\end{thebibliography}
\end{document}